\documentclass[12pt]{article}
\usepackage{amsthm}

\newtheorem{theorem}{Theorem}[section]

\newtheorem{lemma}[theorem]{Lemma}

\newtheorem{corollary}[theorem]{Corollary}
\newtheorem{conjecture}[theorem]{Conjecture}

\newtheorem{example}[theorem]{Example}
\newtheorem{remark}[theorem]{Remark}

\usepackage{xcolor}
\usepackage{comment}
\usepackage{tkz-euclide}
\usepackage{ytableau}
\usepackage{mathtools}
\usepackage{amsfonts}
\usepackage{amsmath}
\usepackage{tikz}

\begin{document}

\author{Galyna Dobrovolska\\
\small Department of Mathematics\\[-0.8ex]
\small Ariel University\\[-0.8ex] 
\small Ariel, Israel\\
\small\tt galdobr@gmail.com}

\title{Combinatorial wall-crossing for the sign representation of prime size}

\maketitle

\begin{abstract}

We present an algorithm to calculate the result of combinatorial wall-crossing at every step starting with the column partition of prime size. This algorithm is confirmed by computer calculations. The output of the algorithm is consistent with Bezrukavnikov's conjecture in the case of the sign representation.

%We establish a new simple explicit description of the Young diagram at each step of the combinatorial wall-crossing algorithm for the rational Cherednik algebra applied to the trivial representation. In this way we provide a short and explicit proof of a key theorem of P. Dimakis and G. Yue. We also present a conjecture on combinatorial wall-crossing which was found using computer experiments.

\end{abstract}

\section{Introduction}

The object of study of this paper is an algorithm for combinatorial wall-crossing for the rational Cherednik algebra. For two positive integers $r \leq s$ let $0<\frac{r^{\prime}_1}{s^{\prime}_1}<...<\frac{r^{\prime}_q}{s^{\prime}_q}<r/s$ be all the rational numbers between $0$ and $r/s$ whose denominator is at most $n$, to be called "walls," i.e. the terms of the $n$-th Farey sequence which are smaller than $r/s$. For a positive integer $e$ let $M_e$ be the generalized Mullineux involution which is a certain extension of the usual $e$-Mullineux involution from the set of $e$-regular partitions to all partitions, see below. For a partition $\alpha$ let $\alpha^t$ stand for the transpose partition. In this paper we study the permutation $\widetilde{M}_{r/s}= M_{s^{\prime}_k}^t \circ \dots \circ M_{s^{\prime}_1}^t$ on partitions of $n$. We think of the involution $M_{s^{\prime}_i}^t$ as "wall-crossing" across the wall $\frac{r^{\prime}_i}{s^{\prime}_i}$ and call $\widetilde{M}_{r/s}$ the {\it combinatorial wall-crossing transformation} to the left of the wall $r/s$.

 Beilinson and Ginzburg (\cite{BG}) were the first to study wall-crossing functors for representations of complex semisimple Lie algebras. Wall-crossing functors for quantized symplectic resolutions appeared as a crucial tool in the work of Bezrukavnikov and Losev (\cite{BL}). Wall-crossing functors are perverse equivalences, inducing bijections between irreducible objects of the corresponding category $\mathcal O$, called combinatorial wall-crossing. It was proved by Losev (\cite{L}) that combinatorial wall-crossing for the rational Cherednik algebra of type A is the same as the Mullineux involution in large positive characteristic.  Wall-crossing was also recently the subject of study of the works by Gorsky and Negut (\cite{GN}) and Su, Zhao, and Zhong (\cite{SZZ}). 
 
 Considering the above permutation $\widetilde{M}_{r/s}$ is motivated by Bezrukavnikov's combinatorial conjecture (see below). This conjecture relates the wall-crossing functors for the rational Cherednik algebra and the monodromy of the quantum connection for the Hilbert scheme of points in the plane (cf. \cite{BO}). Namely, Bezrukavnikov conjectured that certain invariants, which correspond to dimensions of supports of simple representations of the rational Cherednik algebra, of the composition of the above involutions are the same as similar invariants for a composition of much simpler involutions, which do not involve the Mullineux involution in their definition.

In order to state Bezrukavnikov's conjecture, we have to define another operation $\widetilde{M}^{\prime}_{r/s}$ on partitions of $n$ for a given term $r/s$ of the $n$-th Farey sequence. Firstly, consider the operation of concatenation $\cup$ of two partitions, defined as follows: the sequence of rows of the new partition is obtained as the multiset union of the two sequences of rows of the old partitions. Next we represent our partition uniquely as a concatenation $\mu = \nu \cup e \rho$ where each row of $\nu$ is not divisible by $e$ and each row of $e \rho$ is equal to $e$ times the corresponding row of the partition $\rho$. For example, for $e=2$ and $\mu=(4,3,2,1)$ we have $\nu=(3,1)$ and $\rho=(2,1)$. Finally we define $M^{\prime}_e(\mu) = \nu \cup e \rho^t$ and $\widetilde{M}^{\prime}_{r/s} = M_{s^{\prime}_k}^{\prime} \circ \dots \circ M_{s^{\prime}_1}^{\prime}$, where  $0<\frac{r^{\prime}_1}{s^{\prime}_1}<...<\frac{r^{\prime}_k}{s^{\prime}_k}<r/s$ are all the rational numbers between $0$ and $r/s$, whose denominator is at most $n$. 

A partition is called $e$-regular if each of its parts is repeated no more than $e-1$ times. To state Bezrukavnikov's conjecture, we make use of an extension of the Mullineux involution from the set of $e$-regular partitions to the set of all partitions. We will abuse the notation slightly and denote the extension by the same letter $M_e$ and call it the {\it generalized Mullineux involution}. For a partition $\rho$ denote by $e * \rho$ the partition which is obtained by repeating each part of $\rho$ exactly $e$ times. We can write uniquely $\mu = \nu \cup e * \rho$ with $\nu$ regular. Having done this, we define $M_e(\mu) = M_e(\nu) \cup e * \rho^t$. Now we will state Bezrukavnikov's conjecture:

\begin{conjecture} {\rm (R. Bezrukavnikov, cf. Conjecture A.2 in \cite{DY})} For every positive integer $n$, every partition $\lambda$ of $n$, and every term of the Farey sequence $r/s$, the total number of boxes in all rows divisible by $s$ in the two partitions $\widetilde{M}_{r/s}(\lambda)$  and $\widetilde{M}^{\prime}_{r/s}(\lambda^t)$ is the same. 
\end{conjecture}

\begin{remark} Examples verifying this conjecture for $n \leq 5$ are located in the appendix to \cite{DY}.
\end{remark}

\begin{remark} Wall-crossing for the trivial representation was studied in \cite{DY} and \cite{D} and Bezrukavnikov's conjecture was proven for the trivial representation therein. In particular, the paper \cite{D} gives an explicit algorithm for finding the result of combinatorial wall-crossing for the trivial representation at every step, i.e. to the left of any fraction in the Farey sequence.
\end{remark}

\section{Sign representation}

Given positive integers $f_1>\dots>f_m$ and $g_1,\dots,g_m$, we will denote by $(f_1^{g_1}, f_2^{g_2}, \dots, f_m^{g_m})$ the partition of the integer $f_1 g_1 + \dots + f_m g_m$ with $g_i$ parts equal to $f_i$. Let $p$ be a prime number. Below we describe the algorithm to construct each step of combinatorial wall-crossing with respect to the $p$-th Farey sequence starting with the column partition $(1^p)$. 

Let us fix a wall $\frac{m}{p}$. Then the partition to the left of $\frac{m}{p}$ is $(1^p)$ and the partition to the right of $\frac{m}{p}$ is $(p)$. Let $\frac{f}{g}$ be the closest fraction to the right of $\frac{m}{p}$ in the $p$-th Farey sequence. Let $p=(g-1)x + y$ be the division with remainder. Then the Young diagram immediately to the right of $\frac{f}{g}$ is $((g-1)^x,y)$.

The next steps of the algorithm are as follows. We start with the partition $(a^c, b^d)=((g-1)^x, y)$ in the first step of the algorithm above. At each step of the algorithm we take the current two-step partition $(a^c, b^d)$ and send it to the next two-step partition $((a-b)^{c(k+1)+dk}, l^{c+d})$ where $b=k(a-b)+l$ is the division with remainder.

We recover the denominators of the walls where the partition changes to the left of $(a^c, b^d)$ as $a+d$ and to the right of $(a^c, b^d)$ as $c+d+a-b$. The denominator uniquely determines the Farey fraction in the interval $(\frac{m}{p}, \frac{m+1}{p})$, hence from the algorithm we recover all the Farey fractions where the partition changes. At all the intermediate Farey fractions between the ones we recovered above from the algorithm the Young diagram does not change.

We also obtain four equations with four unknowns for $a,b,c,d$ in the partition $(a^c, b^d)$ in terms of the terms of the Farey sequence where the Young diagram changes. Let $b=k(a-b)+l$ be the division with remainder. Let $\frac{f^{\prime}}{g^{\prime}}$, $\frac{f^{\prime \prime}}{g^{\prime \prime}}$, $\frac{f^{\prime \prime \prime}}{g^{\prime \prime \prime}}$ be the walls where the partition changes to the left of $(a^c, b^d)$, to the right of $(a^c, b^d)$, and one more step to the right. Then we have the following system of equations:

\begin{equation}
\label{system}
\begin{cases}
ac+bd=p \\
a+d = g^{\prime} \\
c+d+a-b = g^{\prime \prime} \\
c(k+2)+d(k+1)+a+b-l = g^{\prime \prime \prime}
\end{cases}
\end{equation}

 \

\begin{example}

The following diagram illustrates the above algorithm and the system of equations (\ref{system}) for $p=29$:

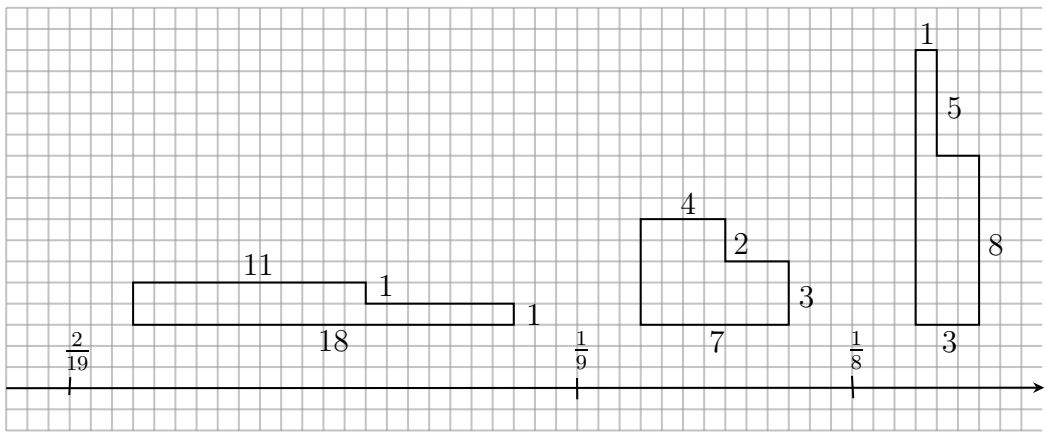
\begin{center}

\tikzset{every picture/.style={line width=0.75pt}} %set default line width to 0.75pt        

\begin{tikzpicture}[x=0.4pt,y=0.4pt,yscale=-1,xscale=1]
%uncomment if require: \path (0,310); %set diagram left start at 0, and has height of 310

%Shape: Grid [id:dp4348674371090142] 
\draw  [draw opacity=0] (-247.8,-325.4) -- (731.91,-325.4) -- (731.91,75.1) -- (-247.8,75.1) -- cycle ; \draw  [color={rgb, 255:red, 155; green, 155; blue, 155 }  ,draw opacity=0.6 ] (-247.8,-325.4) -- (-247.8,75.1)(-227.8,-325.4) -- (-227.8,75.1)(-207.8,-325.4) -- (-207.8,75.1)(-187.8,-325.4) -- (-187.8,75.1)(-167.8,-325.4) -- (-167.8,75.1)(-147.8,-325.4) -- (-147.8,75.1)(-127.8,-325.4) -- (-127.8,75.1)(-107.8,-325.4) -- (-107.8,75.1)(-87.8,-325.4) -- (-87.8,75.1)(-67.8,-325.4) -- (-67.8,75.1)(-47.8,-325.4) -- (-47.8,75.1)(-27.8,-325.4) -- (-27.8,75.1)(-7.8,-325.4) -- (-7.8,75.1)(12.2,-325.4) -- (12.2,75.1)(32.2,-325.4) -- (32.2,75.1)(52.2,-325.4) -- (52.2,75.1)(72.2,-325.4) -- (72.2,75.1)(92.2,-325.4) -- (92.2,75.1)(112.2,-325.4) -- (112.2,75.1)(132.2,-325.4) -- (132.2,75.1)(152.2,-325.4) -- (152.2,75.1)(172.2,-325.4) -- (172.2,75.1)(192.2,-325.4) -- (192.2,75.1)(212.2,-325.4) -- (212.2,75.1)(232.2,-325.4) -- (232.2,75.1)(252.2,-325.4) -- (252.2,75.1)(272.2,-325.4) -- (272.2,75.1)(292.2,-325.4) -- (292.2,75.1)(312.2,-325.4) -- (312.2,75.1)(332.2,-325.4) -- (332.2,75.1)(352.2,-325.4) -- (352.2,75.1)(372.2,-325.4) -- (372.2,75.1)(392.2,-325.4) -- (392.2,75.1)(412.2,-325.4) -- (412.2,75.1)(432.2,-325.4) -- (432.2,75.1)(452.2,-325.4) -- (452.2,75.1)(472.2,-325.4) -- (472.2,75.1)(492.2,-325.4) -- (492.2,75.1)(512.2,-325.4) -- (512.2,75.1)(532.2,-325.4) -- (532.2,75.1)(552.2,-325.4) -- (552.2,75.1)(572.2,-325.4) -- (572.2,75.1)(592.2,-325.4) -- (592.2,75.1)(612.2,-325.4) -- (612.2,75.1)(632.2,-325.4) -- (632.2,75.1)(652.2,-325.4) -- (652.2,75.1)(672.2,-325.4) -- (672.2,75.1)(692.2,-325.4) -- (692.2,75.1)(712.2,-325.4) -- (712.2,75.1) ; \draw  [color={rgb, 255:red, 155; green, 155; blue, 155 }  ,draw opacity=0.6 ] (-247.8,-325.4) -- (731.91,-325.4)(-247.8,-305.4) -- (731.91,-305.4)(-247.8,-285.4) -- (731.91,-285.4)(-247.8,-265.4) -- (731.91,-265.4)(-247.8,-245.4) -- (731.91,-245.4)(-247.8,-225.4) -- (731.91,-225.4)(-247.8,-205.4) -- (731.91,-205.4)(-247.8,-185.4) -- (731.91,-185.4)(-247.8,-165.4) -- (731.91,-165.4)(-247.8,-145.4) -- (731.91,-145.4)(-247.8,-125.4) -- (731.91,-125.4)(-247.8,-105.4) -- (731.91,-105.4)(-247.8,-85.4) -- (731.91,-85.4)(-247.8,-65.4) -- (731.91,-65.4)(-247.8,-45.4) -- (731.91,-45.4)(-247.8,-25.4) -- (731.91,-25.4)(-247.8,-5.4) -- (731.91,-5.4)(-247.8,14.6) -- (731.91,14.6)(-247.8,34.6) -- (731.91,34.6)(-247.8,54.6) -- (731.91,54.6)(-247.8,74.6) -- (731.91,74.6) ; \draw  [color={rgb, 255:red, 155; green, 155; blue, 155 }  ,draw opacity=0.6 ]  ;
%Shape: Polygon [id:ds6310708914989467] 
\draw   (92.2,-45.4) -- (232.2,-45.4) -- (232.2,-25.4) -- (-127.8,-25.4) -- (-127.8,-65.4) -- (92.2,-65.4) -- cycle ;
%Shape: Polygon [id:ds7842342801908124] 
\draw   (492.2,-85.4) -- (492.2,-25.4) -- (352.2,-25.4) -- (352.2,-125.4) -- (432.2,-125.4) -- (432.2,-85.4) -- cycle ;
%Shape: Polygon [id:ds6383239657898885] 
\draw   (632.2,-285.4) -- (632.2,-185.4) -- (672.2,-185.4) -- (672.2,-25.4) -- (612.2,-25.4) -- (612.2,-285.4) -- cycle ;
%Straight Lines [id:da8899073097909573] 
\draw    (-247.8,34.6) -- (730.8,34.1) ;
\draw [shift={(733.8,34.1)}, rotate = 179.97] [fill={rgb, 255:red, 0; green, 0; blue, 0 }  ][line width=0.08]  [draw opacity=0] (10.72,-5.15) -- (0,0) -- (10.72,5.15) -- (7.12,0) -- cycle    ;
%Straight Lines [id:da5267564639157307] 
\draw    (-187.2,24.1) -- (-188.2,41.1) ;
%Straight Lines [id:da7733893904461133] 
\draw    (292,25) -- (292.2,45.1) ;
%Straight Lines [id:da14100811705400051] 
\draw    (552,23) -- (553.2,44.1) ;

% Text Node
\draw (-197,-20.6) node [anchor=north west][inner sep=0.75pt]    {$\frac{2}{19} \ \ \ \ \ \ \ \ \ \ \ \ \ \ \ $};
% Text Node
\draw (285,-21.6) node [anchor=north west][inner sep=0.75pt]    {$\frac{1}{9}$};
% Text Node
\draw (545,-21.6) node [anchor=north west][inner sep=0.75pt]    {$\frac{1}{8}$};
% Text Node
\draw (-27,-94.6) node [anchor=north west][inner sep=0.75pt]    {$11$};
% Text Node
\draw (101,-74.6) node [anchor=north west][inner sep=0.75pt]    {$1$};
% Text Node
\draw (241,-47.6) node [anchor=north west][inner sep=0.75pt]    {$1$};
% Text Node
\draw (44,-22.6) node [anchor=north west][inner sep=0.75pt]    {$18$};
% Text Node
\draw (414.2,-22) node [anchor=north west][inner sep=0.75pt]    {$7$};
% Text Node
\draw (499,-64.6) node [anchor=north west][inner sep=0.75pt]    {$3$};
% Text Node
\draw (437,-114.6) node [anchor=north west][inner sep=0.75pt]    {$2$};
% Text Node
\draw (387,-152.6) node [anchor=north west][inner sep=0.75pt]    {$4$};
% Text Node
\draw (634.2,-22) node [anchor=north west][inner sep=0.75pt]    {$3$};
% Text Node
\draw (678,-113.6) node [anchor=north west][inner sep=0.75pt]    {$8$};
% Text Node
\draw (639,-243.6) node [anchor=north west][inner sep=0.75pt]    {$5$};
% Text Node
\draw (613,-313.6) node [anchor=north west][inner sep=0.75pt]    {$1$};

\end{tikzpicture}

\end{center}

Note that in the above diagram we are initially at the stage of wall-crossing for the sign representation corresponding to the Young diagram $(a^c, b^d)=(18^1, 11^1)$. The next diagram under wall-crossing is $(7^3,4^2)$, and the next after that is $(3^8, 1^5)$. Since we have $a=18, c=1, b=11, d=1$, we can check the equations in the system (\ref{system}). We have $ac+bd=29$. Note also the Farey fractions in the above diagram where the partition changes unde wall-crossing: $\frac{f^{\prime}}{g^{\prime}}=\frac{2}{19}$, $\frac{f^{\prime \prime}}{g^{\prime \prime}}=\frac{1}{9}$, $\frac{f^{\prime \prime \prime}}{g^{\prime \prime \prime}}=\frac{1}{8}$. Now we check the rest of the equations in the system (\ref{system}). We have $a+d=19=g^{\prime}$, which is the $g^{\prime}$-rim of the transposed partition $(18^1,11^1)^t$. We also have $c+d+a-b=1+1+18-11=9=g^{\prime \prime}$, which is the $g^{\prime \prime}$-rim of the partition $(18,11)$ and the $g^{\prime \prime}$-rim of the transpose of the next partition $(7^3, 4^2)^t$, which can be computed as $7+2$. Finally, we can calculate $k$ and $l$ as division with remainder $11=k(18-11)+l$, so $k=1$ and $l=4$. We check the prediction of the algorithm that $(a-b)^{c(k+1)+dk}, l^{c+d})=(7^3, 4^2)$ because $c(k+1)+dk=2+1=3$. Now the length of the $g^{\prime \prime \prime}$-rim of the transposed partition $(7^3,4^2)^t$ is $g^{\prime \prime \prime}=8=2+3+7-4=(c(k+1)+dk)+(c+d)+(a-b)-l=c(k+2)+d(k+1)+a-b-l$.

\end{example}

%{\color{red} Draw pictures of two partitions $(a^c, b^d)$ and $(a-b)^{c(k+1)+dk}, l^{c+d})$ with their rims to illustrate the system of equations and explain that $(c+d+a-b)$-Mullineux transpose maps one to the other}

%\begin{example}
%We illustrate the above algorithm with the example of what happens to the column partition of content $5$ under the sequence of combinatorial wall-crossings across the walls in the $5$-th Farey sequence. It is easy to verify that the algorithm predicts all the image partitions under combinatorial wall-crossing and Farey fractions where they change in this case, and to check equations (\ref{system}):

%$(1,1,1,1,1) \xrightarrow{\frac{1}{5}} (5) \xrightarrow{\frac{1}{4}} (3,2) \xrightarrow{\frac{1}{3}} (1,1,1,1,1) \xrightarrow{\frac{2}{5}} (5) \xrightarrow{\frac{1}{2}} (1,1,1,1,1) \xrightarrow{\frac{3}{5}} (5) \xrightarrow{\frac{2}{3}} (2,2,1) \xrightarrow{\frac{3}{4}} (1,1,1,1,1) \xrightarrow{\frac{4}{5}} (5).$

%\end{example}

%\

%{\color{red} Mention the program that checks this algorithm}

%\

%\

%{\color{red} State the conjecture that the above algorithm gives the combinatorial wall-crossing for the sign representation of prime size at each wall}

\begin{conjecture} The above algorithm gives the result of combinatorial wall-crossing at each step starting with the sign representation of prime size.
\end{conjecture}

The code that checks the above algorithm is available upon request.

Note that the above algorithm is consistent with Bezrukavnikov's conjecture for the sign representation of prime size. Namely for every Farey fraction $\frac{r}{s}$ the operation $\widetilde{M}_{r/s}$ always gives $\widetilde{M}_{r/s}(1^p) = (1^p)$ and the total number of boxes divisible by $s$ in $(1^p)$ is $p$ if $s=p$ and $0$ otherwise.

\subsection*{Acknowledgements}
The author is grateful to Roman Bezrukavnikov for suggesting this problem and for many useful discussions.

\end{document}